\documentclass[11pt]{amsart}
\usepackage{graphicx}
\usepackage[english]{babel}
\usepackage{amsthm}\usepackage[ruled,linesnumbered]{algorithm2e}
\usepackage{float}\usepackage{multirow}
\usepackage{amsmath}\usepackage{cases}
\usepackage{amssymb}
\usepackage{amsfonts}
\usepackage{color}
\usepackage{enumerate}
\usepackage{cite}
\usepackage{hyperref}
\usepackage{stmaryrd}
\usepackage{caption}
\usepackage{subcaption}
\usepackage{mathrsfs}
\usepackage{enumitem}
\usepackage{mathrsfs}

\usepackage{geometry}

\numberwithin{equation}{section}

\allowdisplaybreaks[3]

\newtheorem{thm}{Theorem}[section]
\newtheorem{lemma}[thm]{Lemma}

\theoremstyle{definition}

\theoremstyle{remark}
\newtheorem{remark}[thm]{Remark}

\usepackage{array}

\usepackage{blindtext}

\newcommand{\comment}[1]{}

\usepackage{bm}

\newcommand{\bL}{\bm L}
\newcommand{\bu}{\bm u}

\newcommand{\bv}{\bm v}
\newcommand{\bn}{\bm n}

\newcommand{\bH}{\bm H}

\newcommand{\MJN}[1]{\textcolor{black}{#1}}
\newcommand{\MN}[1]{\textcolor{black}{#1}}
\newcommand{\SG}[1]{\textcolor{black}{#1}}

\title[Extended Lagrange \MN{Finite Element Methods} for Maxwell eigenvalue problem]{An extended Lagrange FEM for the Maxwell eigenvalue problem}

\author[J. Han]{Jiayu Han}
\address{School of  Mathematical Sciences, Guizhou Normal University, Guiyang,  550025 }
\email{hanjiayu@gznu.edu.cn, hanjiayu126@126.com}           
\thanks{This work is partially supported by  the National Natural Science Foundation of
China (under Grants 12361084 and 12001130).}

\begin{document}

\maketitle

\begin{abstract} We construct an extended Lagrange FE space to solve the Maxwell  equation and its eigenvalue problem in $\mathbb R^d$ $(d=2,3)$, which is the sum of the vectorial $p-$order Lagrange FE space ($p\ge1$) and the gradient of the $p+1-$order Lagrange FE space. The two lowest-order methods in 3D adopt slightly less degrees of freedom   than  the second family of the same order edge element methods  in 3D.
     We construct a Cl\'{e}ment interpolant operator to prove the discrete compactness of the FE space and the convergence of  the new methods for both Maxwell equation and its eigenvalue problem.
      For  the
 extended linear Lagrange element method, an average-type curl recovery approach is designed to obtain numerical solution of super-convergence.  In the numerical part, we verify the optimal convergence order for the two lowest-order methods, discuss the upper bound property of numerical eigenvalues and investigate the lower bound property by the average-type curl recovery approach.
\end{abstract}

\parbox{14.5cm}{\noindent\textbf{\small Keywords:~}{\footnotesize Maxwell eigenvalue problem, Lagrange finite element, recovery method, eigenvalue upper/lower  bound.}
}
\thispagestyle{empty}

\section{Introduction}
It is well known that the N\'ed\'elec edge elements \cite{nedelec} can approximate the Maxwell eigenvalue problem well in the canonical discrete form.  We refer the  reader  to  \cite[Section 20]{FE2010}
for an extensive survey.
However, the direct use of Lagrange finite elements for Maxwell eigenvalue problem will produce erroneous solutions on general meshes (see for example \cite{FE2006, FE2010}).
Two basic approaches to obtain stable discrete schemes include introducing the penalization or regularization terms \cite{bonito2011approximation,badia2012nodal, buffa2009solving, duan2019family} and   adopting special meshes
\cite{wong1988combined,powell1977piecewise,worsey1987ann,boffi1,boffi2}. The works by Hu et al. \cite{
hu2021spurious, hu2022partially}  developed finite element methods using partially  discontinuous  Lagrange bases on different meshes.    Wong and Cendes  \cite{wong1988combined} adopted Lagrange element methods to obtain the approximation of Maxwell eigenvalues on Powell-Sabin \cite{powell1977piecewise} meshes in $\mathbb R^2$ or   Worsey-Farin  meshes \cite{worsey1987ann} in $\mathbb R^3$. Recently Boffi et al. \cite{boffi1,boffi2} theoretically justified this numerical fact.

 To enforce the div-free constraint, the Maxwell eigenvalue problem  can be discretized as  mixed forms (c.f. \cite{monk,boffi1,boffi2,CME1999,hiptmair, xiao,gong}).
 However, the canonical mixed form using the Lagrange finite element space is not stable. Our research  aims at proposing an extended Lagrange element method to fill this gap. To do this, we extend the vectorial $p$-order Lagrange element space $\bm L_{h,0}$ with the gradient of its $p+1$-order  space $U_h$ $(p\ge1)$. The extended FE methods  in low order cases have slightly less degrees of freedom than the second family of edge element methods. In order to accomplish the error analysis, we prove the discrete compactness of the extended spaces $\bm L_{h,0}+\nabla U_h$ and the pointwise convergence of the discrete solution operator.
Then  the optimal convergence order of the discrete eigen-pairs is obtained by the Babu\u{s}ka-Osborn theory \cite{babuska}.  

In addition, we propose an average-type  curl recovery scheme for the extended linear Lagrange element method. We can verify  that when the Maxwell eigenfunction is smooth,  the numerical eigenvalue approximates the exact one from upper(see section 5.2). Then this scheme can be used to obtain the supercloseness lower bound of the exact eigenvalues.
For the other work of the numerical lower bound of Maxwell eigenvalues we refer the reader to  \cite{gallistl};  we also refer the reader to \cite{yang,hu,carstensen,you} for  other related early works.

To the best of our knowledge, the proposed methods seem to be the first effective work  using the canonical discrete form and Lagrange element spaces  on general simplicial meshes in both 2D and 3D. It is worth noting that the finite element programs using Lagrange element spaces have been well exploited  on existing softwares. 


We organize this paper as follows. In Section 2,
we introduce the extended Lagrange element spaces, the canonical discrete formulation and its equivalent mixed formulation.
In Section 3, we prove the discrete compactness of the extended Lagrange element space and the pointwisely convergence of the discrete solution operators. Then we derive the optimal convergence order of the discrete eigen-pairs. In Section 4, we design a average type curl recovery scheme for the extended linear Lagrange elements.
In the last section, we   verify numerically the theoretical results using the two lowest order methods in both 2D and 3D.

\section{Extended Lagrange element Discretization}\label{sec-framework}

We consider the
shape regular and affine meshes $\mathcal T_{h} = \{ K\}$ that partition the  simply-connected Lipschitz domain $\Omega \subset \mathbb{R}^d$ ($d=2,3$)  into triangles or tetrahedra in $\mathbb R^d$.
     The extended Lagrange finite element spaces in  this paper is given as
\begin{align}
\bm V_{h  }&=\bm L_{h,0}+\nabla U_{h  }\label{4.4a}
\end{align}
where \begin{align*}
\bm L_h&=\{\bm q_h\in   (H^1(\Omega))^d|\bm q_h|_ K\in (P_ {p}(K))^d,~\forall  K\in\mathcal T_{h}\},\\
\bm L_{h,0}&=\{\bm q_h\in   \bm L_h|\bm q_h\times \bm n=0~\text{on}~\partial\Omega\},\\
U_{h  }&=\{q_h\in  H^1_0(\Omega)|q_h|_ K\in P_ {p+1}(K),~\forall  K\in\mathcal T_{h}\},\quad   p\ge 1.
\end{align*}
The Lagrange FE space $\bm L_{h,0}$ provides the approximation property of piecewise $p$-degree polynomials and
the gradient space $\nabla U_h$ provides the tangential continuity.

Consider the Maxwell eigenvalue problem : Find $(\lambda,\bm u)\in\mathbb R\times\bm H_0(\mathrm{curl},\Omega)\backslash \{0\}$ such that
 \begin{equation} \label{equ:Ceig}
 (\mathrm{curl} \bu, \mathrm{curl} \bv) = \lambda (\bu,\bv), ~~~ \forall \bv \in \bH_0(\mathrm{curl}, \Omega).
 \end{equation}
 We introduce the space
 \begin{equation*}
 \bH_0(\mathrm{curl}, \Omega):=\{\bv \in \bL^2(\Omega): \mathrm{curl} \bv \in \bL^2(\Omega) \text{~and~} \bv \times \bn=0 \text{~on~} \partial \Omega\},
 \end{equation*}
 equipped with the standard norm $\|\cdot\|_{\mathrm{curl}}$.
 Here $\bn$ is an  unit outward normal vector on $\partial \Omega$.

The Maxwell eigenvalue problem can be also written as: Find $(\lambda,\bm u)\in\mathbb R\times\bm H_0(\mathrm{curl},\Omega)\backslash \{0\}$ and ${\rho}\in H^1_0(\Omega)$ such that
\begin{align}\label{2.10a}
\begin{aligned}
(\mathrm{curl}\bm u,\mathrm{curl}\bm{v})+(\nabla {\rho},\bm{v})&=\lambda( \bm u,\bm v),~~\forall \bm {v}\in \bm H_0(\mathrm{curl},\Omega),\\
(\bm u,\nabla q)&=0,~~\forall q\in H_0^1(\Omega).
\end{aligned}
\end{align}

 Accordingly, the extended Lagrange  finite element method with
 \MJN{respect to} the FE space $\bm V_h \subset \bH_0(\mathbf{curl},\Omega)$ is to find
 $\bu_h \in \bm V_h \backslash \{0\}$ and nonzero $\lambda_h \in \mathbb{R}$ such that
 \begin{equation} \label{equ:Deig}
 (\mathrm{curl} \bu_h, \mathrm{curl} \bv_h) = \lambda_h(\bu_h,\bv_h), ~~~ \forall \bv_h \in \bm V_h
 \end{equation}
An equivalent discrete form of the   eigenvalue problem (\ref{equ:Deig}) is given by:
Find $(\lambda_h,\bm u_h,  \rho_h)\in \mathbb{R}\times \bm V_{h  }\times  U_{h  }$
 with $\bm u_h\neq0$ such that
\begin{align}\label{2.9}
\begin{aligned}
(\mathrm{curl}\bm u_h,\mathrm{curl}\bm{v}_h)+(\nabla  {\rho}_h,\bm{v}_{h  })&=\lambda_h( \bm u_{h  },\bm v_{h  }),~~\forall \bm{v}_{h  }\in  \bm V_{h  },\\
(\bm u_{h  },\nabla q)&=0,~~\forall q\in U_{h  }.
\end{aligned}
\end{align}
It is easy to see that $\rho_h=0$.


\section{Error analysis}
First we shall construct a Cl\'ement interpolant operator \cite{clement} adapted to the boundary condition $\bm v\times \bm n=0$ on $\partial\Omega$.
 Let $\Delta_i$ denote the union set of elements sharing the $i$th node and $\widehat{\Delta}_i$ denote its reference element satisfying the linear transformation $\bm x=F_{\Delta_i}(\widehat {\bm x})=B_{\Delta_i}\widehat{\bm x}+b_{\Delta_i}$.
 Define the operator $\widehat{\bm R}_i:L^1(\widehat\Delta_i)\rightarrow \bm P_k(\widehat\Delta_i)$ satisfying
\begin{align}
  \int_{\widehat\Delta_i}(\widehat{\bm R}_i\widehat{\bm v}-\widehat{\bm v})\cdot \widehat{\bm p}=0,\quad \forall \widehat{\bm p}\in\bm P_k(\widehat\Delta_i).
\end{align}
For ${\bm v}\in \bm L^1(\Omega)$ we define $\Pi_h: \bm L^1(\Omega)\rightarrow \bm L_h$ as follows:
\begin{align}
   \Pi_h \bm v=\sum_{i=1}^{N_h}{\bm R}_i{\bm v}(\bm x)\theta_i
\end{align}
where ${ R}_i{\bm v}(\bm x):=\widehat{\bm R}_i\widehat{\bm v}(\widehat{\bm x})$, $N_h=\dim \bm L_h$ and $\theta_i$ is the Lagrange basis corresponding to the $i$th node in the mesh $\mathcal T_h$.

Hereafter we denote the standard Sobolev space $\bm W^{s,q}(D)$ on a bounded Lipschitz domain $D$ equipped with the norm $\|\cdot\|_{s,q,D}$. Let $\bm H^s(D):=\bm W^{s,2}(D)$. 
In this paper, we use the symbol  $x \lesssim y$ ($x \gtrsim y$) to mean $x \le Cy$ ($x \ge Cy$) for a constant $C$ that is independent of the mesh size  and may be different at different occurrences.

Let $\bm v\in\bm H^s(\Delta_i)$ with $0\le m\le s\le p+1$. Pick up any $K\subset\Delta_i$.
By the scaling argument and norm equivalence
\begin{align}\label{3.3}
  |{\bm R}_i{\bm v}-{\bm v}|_{m,K}&\lesssim h_{\Delta_i}^{\frac d 2-m}\inf_{\widehat{\bm q}\in \bm P_k(\widehat K)}|(\widehat{\bm R}_i-I)(\widehat{\bm v}+\widehat{\bm q})|_{m,\widehat \Delta_i}\lesssim h_{\Delta_i}^{\frac d 2-m}|\widehat{\bm v}|_{s,\widehat \Delta_i}\lesssim h_{\Delta_i}^{s-m}|{\bm v}|_{s, \Delta_i}.
\end{align}
Assume that the FE nodes $a_{i_1},\cdots,a_{i_\alpha}$ belong to the element $K$. By inverse estimate we derive
\begin{align}
|{\bm R}_{i_1}{\bm v}-{\bm \Pi}_h{\bm v}|_{m, K}&\le\sum_{j=2}^\alpha\|\theta_{i_j}({\bm R}_{i_1}{\bm v}(a_{i_j})-{\bm R}_{i_j}{\bm v}(a_{i_j}))\|_{m, K}\nonumber\\
  &\lesssim h_{\Delta_i}^{\frac d 2-m}\sum_{j=2}^\alpha\|{\bm R}_{i_1}{\bm v}-{\bm R}_{i_j}{\bm v}\|_{\infty, K}\nonumber\\
  &\lesssim h_{\Delta_i}^{-m}\sum_{j=2}^\alpha\|{\bm R}_{i_1}{\bm v}-{\bm R}_{i_j}{\bm v}\|_{0, K}\nonumber\\
  &\lesssim  h_{\Delta_i}^{s-m}|{\bm v}|_{s, \Delta_i}~by~\eqref{3.3}.
\end{align}
The combination of the above two estimates gives
\begin{align}\label{4.10}
  |{\bm \Pi_h}{\bm v}-{\bm v}|_{m,K}&\le |{\bm R}_{i1}{\bm v}-{\bm v}|_{m, K}+|{\bm R}_{i1}{\bm v}-{\bm \Pi}_h{\bm v}|_{m, K}
  \lesssim h_{\Delta_i}^{s-m}|{\bm v}|_{s, \Delta_i}.
\end{align}

We divide the Lagrange bases into two classes: the interior nodal bases $\{\theta_i\}_{i=1}^{N_{h0}}$ and the boundary nodal bases $\{\theta_i\}_{i=1+N_{h0}}^{N_{h}}$.
For ${\bm v}\in \bm L^1(\Omega)$ we define $\Pi_{h,0}: \bm L^1(\Omega)\rightarrow \bm L_{h,0}$ as follows:
\begin{align}
   \Pi_{h,0} \bm v=\sum_{i=1}^{N_{h0}}{\bm R}_i{\bm v}(a_i)\theta_i+\sum_{i=1+N_{h0}}^{N_{h}}J{\bm v}( a_i)\theta_{i}
\end{align}
where
\begin{subequations}
  \begin{align}
    J\bm{v}(a_i) &= \bigl(\bm{R}_i\bm{v}(a_i) \cdot \bm{n}|_{f_1}(a_i)\bigr)\bm{n}|_{f_1}(a_i), &
    \bm{n}|_{f_1}(a_i) = \bm{n}|_{f_2}(a_i) \text{ and } a_i \in f_1 \cap f_2 \subset \partial \Omega, \label{t11} \\
    J\bm{v}(a_i) &= 0, &
    \bm{n}|_{f_1}(a_i) \neq \bm{n}|_{f_2}(a_i) \text{ and } a_i \in f_1 \cap f_2 \subset \partial \Omega, \label{t2} \\
    J\bm{v}(a_i) &= \bigl(\bm{R}_i\bm{v}(a_i) \cdot \bm{n}|_{f}(a_i)\bigr)\bm{n}|_{f}(a_i), &a_i \in f \subset \partial \Omega \text{ and } a_i \notin \text{any other face}, \label{t3}
  \end{align}
  \end{subequations}

$f$ is an face of $\mathcal T_h$ in $\mathbb R^3$ (or an edge in $\mathbb R^2$), and $f_1$ and $f_2$ are two different faces in $\mathbb R^3$ (or edges in $\mathbb R^2$).
\begin{lemma}For $\bm v\in \bm H^s(\Omega)\cap \bm H_0(\mathrm{curl},\Omega)$ with  $p+1\ge s>1/2$ there holds
\begin{align}
  |{\bm \Pi_{h,0}}{\bm v}-{\bm v}|_{m,K} \lesssim h_{\Delta_i}^{s-m}|{\bm v}|_{s, \Delta_i}~for~0\le m\le s.
\end{align}
\end{lemma}
\begin{proof}
Let  the element $K$ satisfy  $\overline K\cap\partial \Omega\ne\emptyset$ and there are $\beta$ nodes ${i_1},\cdots,{i_\beta}$ of $K$ on $\partial \Omega$.
For any integer $j\in[1,\beta]$  and $a_{i_j}\in\partial \Omega\cap \partial K$ we have
\begin{align*}
({\bm \Pi}_{h}{\bm v}-{\bm \Pi}_{h,0}{\bm v})(a_{i_j})=
\begin{cases}\sum_{j=1}^\beta \sum_{k=1}^{d-1}{\bm R}_{i_j}{\bm v}( a_{i_j})\cdot  \bm\tau_{k}(a_{i_j})\bm\tau_{k}(a_{i_j}),&\text{if $a_{i_j}$ satisfies \eqref{t11} and \eqref{t3}}\\
 {\bm R}_{i_j}{\bm v}( a_{i_j}) ,&\text{if $a_{i_j}$ satisfies \eqref{t2}}
\end{cases}
\end{align*}
where  $\bm\tau_{k}(a_{i_j})$ is the $k$th   tangent direction  at the node $a_{i_j}$.
Pick up $\sigma>0$ such that $\sigma+1/2\le s$. Note that ${\bm v}\times \bm n|_{\partial \Omega}=0$ then
\begin{align}
 |{\bm \Pi}_{h,0}{\bm v}-{\bm \Pi}_h{\bm v}|_{m, K}&\lesssim \sum_{j=1}^\beta \sum_{k=1}^{d-1}||{\bm \Pi}_{h,0}{\bm v}( a_{i_j})-{\bm \Pi}_h{\bm v}( a_{i_j})|\|\theta_{i_j}\|_{m, K}\nonumber\\
 &\lesssim h_{\Delta_i}^{\frac {1} 2-m}\sum_{j=1}^\beta \|{\bm R}_{i_j}{\bm v}-\bm v\|_{0,\partial K\cap \partial \Omega}\nonumber\\
 &\lesssim h_{\Delta_i}^{\frac {d} 2-m}\sum_{j=1}^\beta\inf_{\bm q\in P_k(\widehat K)}\|(\widehat{\bm R}_{i_j}-I)(\widehat{\bm v}+{\bm q})\|_{1/2+\sigma, \widehat K}\nonumber\\
  &\lesssim h_{\Delta_i}^{s-m}|{\bm v}|_{s, \Delta_i}.
\end{align}
This together with \eqref{4.10} yields
\begin{align}
  |{\bm \Pi_{h,0}}{\bm v}-{\bm v}|_{m,K}&\le |{\bm \Pi}_{h}{\bm v}-{\bm v}|_{m, K}+|{\bm \Pi}_{h,0}{\bm v}-{\bm \Pi}_h{\bm v}|_{m, K} \lesssim h_{\Delta_i}^{s-m}|{\bm v}|_{s, \Delta_i}.
\end{align}

\end{proof}

The Hodge operator is a  useful tool in our error analysis.
It is defined as $ H \bm g\in \bm{H}_0(\mathrm{curl},\Omega)$ and $\varrho\in H^1_0(\Omega)$ for $\bm g\in \bm{H}_0(\mathrm{curl},\Omega)$ such that
\begin{align}\label{2.5s}
\begin{aligned}
    &(\mathrm{curl} H \bm g,\mathrm{curl} \bm v)+  ( \nabla \varrho, \bm v)=(\mathrm{curl}\bm g,\mathrm{curl}\bm v),~~\forall \bm{v}\in \mathbf{H}_0(\mathrm{curl},\Omega),\\
     &( \nabla q, H \bm g)=0,~~\forall q\in H^1_0(\Omega).
\end{aligned}
\end{align}

 We introduce the following   curl-conforming  element spaces \cite{nedelec}:
\begin{align}
\widetilde{\bm V}_{h  } &:= \{\bm v_{h  } \in \bm H_0(\mathrm{curl}, \Omega) |\bm v_{h  }|_ K \in (P_{p}(K))^d, \forall  K\in\mathcal T_{h}\}\supset \bm V_h,\\
\bm X_{h  } &:= \{\bm v_{h  } \in {\bm V}_h  |(\bm v_{h  },\nabla q)=0,  \forall q\in U_{h  }\},\\
\bm X &:= \{\bm v \in  \bm H_0(\mathrm{curl}, \Omega) |(\bm v,\nabla q)=0,  \forall q\in H^1_0(\Omega)\},
\end{align}

Let  $\lambda^{Dir}$ be the minimum singularity exponents for the Dirichlet Laplace operators,  and satisfy
\begin{align}
\left\{\psi\in H^1_0(\Omega):\Delta \psi\in L^2(\Omega),~\right\}\subset \bigcap\limits_{s<\lambda^{Dir}}H^{1+s}(\Omega).
\end{align}

We give the   interpolation error estimates corresponding to the above finite element spaces. 
\begin{lemma}[Theorem 5.41 in \cite{monk}] \label{l4.3}{Let $\bm r_{h  }$ be the  edge element interpolation associated with $\widetilde{\bm V}_{h  }$ and $\bm v,\mathrm{curl}\bm v\in \bm H^s(\Omega)$$(p\ge s>1/2)$ then
 \begin{align}\label{b0}
\|r_{h  } \bm v- \bm v\|_{\mathrm{curl}}\lesssim
 h^{s}(\|\bm v\|_{{s,\Omega}}+ \|\mathrm{curl}\bm v\|_{s,\Omega}).
\end{align}
 Let $\bm v\in \bm X\subset \bm H^{r_D}(\Omega)$ for  $r_D:=\lambda^{Dir}-\varepsilon$ with sufficiently small  $\varepsilon>0$ and $\mathrm{curl}\bm v\in \mathrm{curl}\widetilde{\bm V}_{h  }$ then}
\begin{align}\label{b1}
\|r_{h  } \bm v- \bm v\|_{}\lesssim
 h^{r_D}(\|\bm v\|_{{r_D,\Omega}}+ \|\mathrm{curl}\bm v\|_{}).
\end{align}
    \end{lemma}

 \begin{lemma} Let $\bm v_{h  }\in {\bm X}_{h  }$ then $H\bm v_{h  }\in \bm X$ such that $\mathrm{curl}\bm v_{h  }=\mathrm{curl}H\bm v_{h  }$ and
\begin{align}
\|H\bm v_{h  } \|_{r_D,\Omega}&\lesssim\|\mathrm{curl}\bm v_{h  } \|_{},\label{4.15}\\
 \|H\bm v_{h  }-\bm v_{h  }\|_{}&\le  \|r_{h  } H\bm v_{h  }- H\bm v_{h  }\|_{}.
 \label{4.16}%
\end{align}
    \end{lemma}
    \begin{proof}The proof follows Lemma 7.6 in \cite{monk} or  Lemma 4.5 in \cite{hiptmair}. The regularity estimate \eqref{4.15} is due to $\bm X\subset  H^{r_D}(\Omega)$(c.f. Remark 3.8 in \cite{amrouche}). According to the definition of $ H$, we have   $\mathrm{curl}  H \bm v_{h  }=\mathrm{curl}\bm v_{h  }$.
By the commuting property of $r_h$ we have $\mathrm{curl}(r_{h  } H \bm v_{h  })=\mathrm{curl} \bm v_{h  }$.
 This implies  $r_{h  } H \bm v_{h  }- \bm v_{h  }\in \nabla U_{h  }$ (see Section 7.2.1 in \cite{monk}).
According to  $( H \bm v_{h  }- \bm v_{h  },\bm q)=0,~\forall \bm q\in \nabla U_{h  }$, we have $( H \bm v_{h  }- \bm v_{h  },r_{h  } H \bm v_{h  }- \bm v_{h  })=0$.
Hence using Lemma 3.2 we deduce
\begin{align*}
( H \bm v_{h  }- \bm v_{h  }, H \bm v_{h  }- \bm v_{h  })&=( H \bm v_{h  }- r_{h  }H\bm v_{h  }, H \bm v_{h  }- \bm v_{h  })\\
&\lesssim
 \| H \bm v_{h  }- \bm v_{h  }\|_{}  (h^{r_D} \| H\bm v_{h  }\|_{r_0,\Omega}+h\|\mathrm{curl}{\bm v}_{h  }\|_{}).
\end{align*}
Then the estimate \eqref{4.16} is obtained by the above estimate and \eqref{4.15}.
\end{proof}

To analyze the convergence of the discretization \eqref{2.9},
we consider the   source problem with div-free constraint: Find $T\bm f\in\bm H_0(\mathrm{curl},\Omega)$ and $S\bm f\in H^1_0(\Omega)$ such that
\begin{align}\label{p2}
\begin{aligned}
(\mathrm{curl}T\bm f,\mathrm{curl}\bm{v})+( \nabla S\bm f, \bm v)&=( \bm f,\bm v),~~\forall \bm{v}\in \bm H_0(\mathrm{curl},\Omega),\\
( \nabla q, T\bm f)&=0,~~\forall q\in H^1_0(\Omega).
\end{aligned}
\end{align}
Its extended Lagrange FE discretization is to seek $T_{h  }\bm f\in \bm V_{h  }$ and $S_h\bm f\in U_{h  }$:
\begin{align}\label{p3}
\begin{aligned}
(\mathrm{curl}T_{h  }\bm f,\mathrm{curl}\bm{v}_{h  })+( \nabla S_h\bm f, \bm v_{h  })&=( \bm f,\bm v_{h  }),~~\forall \bm{v}_{h  }\in \bm V_{h  },\\
(  \nabla q, \bm T_{h  }\bm f)&=0,~~\forall q\in U_{h  }.
\end{aligned}
\end{align}
It is easy to verify that $S\bm f=S_h \bm f=0$ for any $\bm f \in \bm X$. Let $\bm f=0$ then $T_h\bm f\in \nabla U_h$ and so $T_h\bm f=0$. Therefore the problem \eqref{p3} is uniquely solvable.

Let  $\lambda^{Neu}\in(1/2,1]$ be the minimum singularity exponents for the   Neumann Laplace operators,  and satisfy
\begin{align}
\left\{\psi\in H^1(\Omega):\Delta \psi\in L^2(\Omega),~\frac{\partial \psi}{\partial n}=0~on~\partial\Omega,\int_\Omega\psi dx=0\right\}\subset \bigcap\limits_{s<\lambda^{Neu}}H^{1+s}(\Omega).
\end{align}
\begin{lemma}[Corollary 6.4  in \cite{costabel}]\label{ll1} There holds the following decomposition for $T\bm f$ in the equation \eqref{p2}:

\begin{align}
T\bm f=\bm w_0+\nabla \phi\text{ with $\bm w_0\in\bm H^{1+\lambda^{Neu}-\varepsilon_1}(\Omega)$ and $\phi\in H^1_0(\Omega)\cap H^{1+\lambda^{Dir}-\varepsilon_2}(\Omega)$}\label{4.17}%
\end{align}
where $\varepsilon_1,\varepsilon_2>0$   arbitrarily approach to 0.
\end{lemma}
Since $\mathrm{curl}^2T\bm f=- \nabla S\bm f +\bm f$, by the above lemma we have   $T\bm f, \mathrm{curl}T\bm f\in \bm H^{r}(\Omega)$ for  some $r\ge r_0$ with $r_0$ given by the following theorem.
\begin{thm}\label{t1}
Let  $\bm f\in \bm L^2(\Omega)$, $T\bm f\in \bm H^{r}(\Omega)$ and $\bm w_0\in \bm H^{r+1}(\Omega)$ for $p\ge r\ge r_0:=\min(\lambda^{Dir},\lambda^{Neu})-\varepsilon$ with sufficiently small  $\varepsilon>0$ then
\begin{align}\label{2.1}
\|(  T  -  T_h) \bm f\|_{\mathrm{curl}}&    \lesssim h^r+\inf\limits_{v_h\in  U_h}|S \bm f- v_h|_{1,\Omega};
 \end{align}
 further assume $T  \bm f\in  \bm H^{s}(\Omega)$  for $r\le s\le r_0+r\le r_0+p$ and $\mathrm{div} \bm f=0$ then
 \begin{align}
\|T\bm f-T_h\bm f\| &\lesssim h^{s}.\label{2.1a}
 \end{align}

\end{thm}
\begin{proof}
The combination of \eqref{p2} and \eqref{p3} leads to
\begin{align*}
 (\mathrm{curl}(T_{  }-T_{h  })\bm f,\mathrm{curl}\bm{v}_{h  })+( \nabla (S-S_h)\bm f, \bm v_{h  })&=0,\quad\forall \bm{v}_{h  }\in \bm V_{h  }.
\end{align*}
It follows that
\begin{align*}
 (\mathrm{curl}(\bm v_h-T_{h  }\bm f),\mathrm{curl}(\bm v_h-T_{h  }\bm f))&= (\mathrm{curl}(\bm v_h-T\bm f),\mathrm{curl}(\bm v_h-T_{h  }\bm f))\\
 &\quad-( \nabla (S-S_h)\bm f, \bm v_h-T_{h  }\bm f),\quad\forall \bm{v}_{h  }\in \bm X_{h  }.
\end{align*}
Hence  we have from the Poinc\'are inequality in $\mathbf X_h$ and $\bm V_h=\bm X_h\bigoplus \nabla U_h$
\begin{align*}
\|\mathrm{curl}(\bm v_h-T_{h  }\bm f)\|_{}\lesssim \|\mathrm{curl}(T\bm f-\bm v_h)\|_{}+|(S-S_h)\bm f|_{1,\Omega},\quad\forall \bm{v}_{h  }\in \bm V_{h  }.
\end{align*}
Taking $\bm v_h$ as $\mathbf\Pi_{h,0}\bm w_0$, this together with the triangular inequality infers that
\begin{align*}
\|\mathrm{curl}(T_{  }\bm f-T_{h  }\bm f)\|_{}\lesssim \|\mathrm{curl}(\bm w_0-\mathbf\Pi_{h,0}\bm w_0)\|_{}+|(S-S_h)\bm f|_{1,\Omega}
\lesssim h^{r}+|(S-S_h)\bm f|_{1,\Omega}.
\end{align*}
Let $P_{h  }$ be the orthogonal projection onto $\widetilde{\bm V}_{h  }$ such that for any $\bm v\in\bm H_0(\mathrm{curl},\Omega)$
\begin{align}
(\bm v-P_{h  }\bm v,\bm q)+(\mathrm{curl}(\bm v-P_{h  }\bm v),\mathrm{curl}\bm q)=0,~\forall\bm q\in \widetilde{\bm V}_{h  }.
\end{align}
By the triangular inequality and  the   discrete Poinc\'{a}re inequality we get
\begin{align*}
\|T_{  }\bm f-T_{h  }\bm f\|_{}&\lesssim \|T_{  }\bm f-P_hT\bm f\|_{}+  \|\mathrm{curl}(P_hT\bm f-T_h\bm f)\|_{}\\
&\lesssim \|T_{  }\bm f-P_hT\bm f\|_{\mathrm{curl}}+  \|\mathrm{curl}(T\bm f-T_h\bm f)\|_{}.
\end{align*}
This together with the above estimate and   \eqref{b0}  leads to the conclusion \eqref{2.1}. 
We have the following Helmholtz decomposition
$r_hT\bm f- T_h\bm f=\bm w_0^c +\nabla q_h$ with $\bm w_0^c\in \widetilde{\bm V}_h,(\bm w_0^c,\nabla q)=0,  \forall q\in U_{h  }$ and $q_h\in U_h$.
For $\bm f\in\bm L^2(\Omega)$ with $\mathrm{div}\bm f=0$,  we start as follows
\begin{align}\label{2.41s}
\|T\bm f-T_h\bm f\|^2&=\left(T\bm f-T_h\bm f,  T\bm f-r_hT\bm f+r_hT\bm f-\bm T_h\bm f\right)\nonumber\\
&\lesssim \|T\bm f-r_hT\bm f\|\|T\bm f-T_h\bm f\|+{|(T\bm f-T_h\bm f,\bm w^c_0)|}.
\end{align}
Note that $S_hH\bm w^c_0=SH\bm w^c_0=0$. Using Lemmas 3.2-3.4, the second term at the right-hand side is estimated as follows:
\begin{align}\label{2.45}
|(T\bm f-T_h\bm f,\bm w^c_0)|&\lesssim|(T\bm f-T_h\bm f,\bm w^c_0- H\bm w^c_0)|+|(T\bm f-T_h\bm f, H\bm w^c_0)|\nonumber\\
&\lesssim h^{r_0}\|\mathrm{curl}\bm w^c_0\|\|T\bm f-T_h\bm f\|+|(T\bm f-T_h\bm f, H\bm w^c_0)|\text{ by \eqref{4.16}, \eqref{b1} and \eqref{4.15}}
\end{align}
where
\begin{align}\label{2.48}
&|(T\bm f-T_h\bm f, H\bm w^c_0)|=(\mathrm{curl}(T\bm f-T_h\bm f),\mathrm{curl}(T H\bm w^c_0-\mathbf \Pi_{h,0}T H\bm w^c_0))\nonumber\\
&\quad\lesssim h^{r_0}\|H\bm w^c_0\|\|\mathrm{curl}(T\bm f-T_h\bm f)\|\nonumber\text{ by Lem. 3.1 and \eqref{4.17}}\\
&\quad\lesssim h^{r_0}(h^{r_D}\|\mathrm{curl}\bm w^c_0\|+\|\bm w^c_0\|)\|\mathrm{curl}(T\bm f-T_h\bm f)\|\nonumber\text{ by \eqref{4.16}}\\
&\quad\lesssim (h^{r_0}\|\mathrm{curl}(r_hT\bm f-T_h\bm f)\|_{}+\|r_hT\bm f-T_h\bm f\|_{})\|\mathrm{curl}(T\bm f-T_h\bm f)\|h^{r_0}.
\end{align}
The substitution of (\ref{2.48}) and (\ref{2.45}) into \eqref{2.41s} gives
\begin{align}\label{2.48a}
\|T\bm f-T_h\bm f\| &\lesssim h^{r_0}\|\mathrm{curl}(T\bm f-T_h\bm f)\|+h^{r_0}\|\mathrm{curl}(r_hT\bm f-T\bm f+T\bm f-T_h\bm f)\|_{}+\|r_hT\bm f-T\bm f\|_{}\nonumber\\
&\lesssim h^{r_0+r}+h^{s}~\text{by \eqref{2.1}  and Lemma 3.1},
\end{align}
where we have used $\|r_hT\bm f-T\bm f\|_{}\lesssim h^s\|T\bm f\|_{s,\Omega}$ (see Theorem 8.15 in \cite{monk}).
This completes the proof.
\end{proof}

Let ${\mathcal{H}}$ be a sequence of the mesh sizes converging to 0.
\begin{lemma}(Discrete compactness property) Any sequence $\{\bm v_{h}\}_{h\in\mathcal{H}}$ with $\bm v_{h}\in \bm X_{h  }$ that is uniformly bounded w.r.t $\|\cdot\|_{\mathrm{curl}}$ contains a subsequence that converges strongly in $\bm L^2(\Omega)$.
\end{lemma}
\begin{proof}

  Let $\{\bm v_{h  }\}_{h\in\mathcal{H}}$  with $\|\bm v_{h  }\|_{\mathrm{curl}}\le M$ for a positive constant $M$. It is trivial to  assume that the sequece $h_i\in\mathcal{H}$ satisfies $h_i\rightarrow0$    as $i\to\infty$. By \eqref{4.16}, \eqref{b1} and \eqref{4.15}, $\|H{\bm v}_{h_i}- {\bm v}_{h_i}\|\lesssim h^{r_0} \|\mathrm{curl}\bm v_{h_i}\|\to 0$ as $i\to\infty$.
  Note that
  \begin{align*}
    \|\mathrm{curl}H{\bm v}_{h_i}\|=\|\mathrm{curl}{\bm v}_{h_i}\|\le M.
  \end{align*}
This means that  $\{H{\bm v}_{h_i}\}$ is bounded in $\bm H(\mathrm{curl},\Omega)$. Since $\bm X$ is compactly imbedded into $\bm L^2(\Omega)$, there is a subsequence of $\{H{\bm v}_{h_i}\}$ converging to some $\bm v_0$ in $\bm L^2(\Omega)$. Hence   a subsequence of $\{{\bm v}_{h_i}\}$ will converge to  $\bm v_0$ in $\bm L^2(\Omega)$ as well.
\end{proof}
The following uniform convergence can be derived from the discrete compactness property of $\bm X_{h  }$.
\begin{thm} 
There holds the uniform convergence
\begin{eqnarray*}
& \|{T}_{h  }-{T}\|_{\bm L^2(\Omega)\to \bm L^2(\Omega)}\rightarrow0,~h\rightarrow0.
\end{eqnarray*}
\end{thm}
\begin{proof}
Since   $\cup_{h\in\mathcal{H}}  U_{h  }$ is dense in   $  H^1_0(\Omega)$, we deduce from  Theorem \ref{t1} that $  T_{h  }$ converges to $  T$ pointwisely in {$\bm L^2(\Omega)$}. Thanks to  the discrete compactness
 of $\bm X_{h  }$, $\cup_{h\in\mathcal{H}}T_{h  }Ba$ is a relatively compact set in $\bm L^2(\Omega)$ where $Ba$ is the unit ball in $\bm L^2(\Omega)$. In fact, Let us choose any sequence $\{\bm v_{h  }\}_{h\in\mathcal{H}}\subset Ba$. Note that $T$ is compact from $\bm X$ to $\bm L^2(\Omega)$ then $\{T\bm v_{h  }\}_{h\in\mathcal{H}}$ is a relatively compact set in $\bm L^2(\Omega)$. Hence it holds the collectively compact convergence ${T}_{h  }\rightarrow {T}~in~\bm  L^2(\Omega) \text{ as}~h\rightarrow0$. Noting $T,T_{h  }: \bm L^2(\Omega) \rightarrow \bm L^2(\Omega)$ are self-adjoint, due to Proposition 3.7 or Table 3.1
in \cite{chatelin}  the assertion is valid.
\end{proof}

We are in a position to prove the error estimate for the numerical eigenvalues.
\begin{thm}
Let $\lambda_{h  }$ be an eigenvalue of  \eqref{equ:Deig} converging to  the eigenvalue $\lambda$ of \eqref{2.10a}. Under the condition of theorem 3.5, let 
 $M(\lambda)\subset \{{\bm u}\in \bm H^{r}(\Omega): \bm u\in \bm H^{r+1}(\Omega)+\nabla H_0^1(\Omega) \}$ for some $r\in [r_0,p]$   then
\begin{align}\label{4.19}
&|\lambda-{\lambda}_{h}| \lesssim h^{2r},
\end{align}
where   ${M}(\lambda)$    denotes the space spanned by all eigenfunctions  corresponding to the eigenvalue $\lambda$.
\end{thm}
\begin{proof} Let $\lambda=\lambda_k$ be the $k$th Maxwell eigenvalue and $\dim M(\lambda)=q$.
From   
 Theorem 7.3 
in \cite{babuska} we get
\begin{eqnarray}
&&|\lambda-{\lambda}_{h}| \lesssim \sum\limits_{i,j=k}^{k+q-1}| ((T-T_{h})\bm\phi_{i},\bm\phi_{j})|
+\|(T-T_{h}) |_{M(\lambda)} \|_{\bm L^2(\Omega)\to \bm L^2(\Omega)}^2,\label{l1}
\end{eqnarray}
where $\bm \phi_{k},\cdots,\bm \phi_{k+q-1}$ are  a set of  basis functions for $M(\lambda)$.
Note that  $S\bm f=S_h\bm f=0$  for any $\bm f\in \bm X$ in (\ref{p2}) and (\ref{p3}). 
 Hence the following Garlerkin orthogonality holds:
 \begin{align*}
  ((T-T_{h})\bm\phi_{i},\bm \phi_{j})=(\mathrm{curl}(T-T_{h})\bm\phi_{i},\mathrm{curl}T\bm \phi_{j})
  =(\mathrm{curl}(T-T_{h})\bm\phi_{i},\mathrm{curl}(T-T_{h})\bm \phi_{j}).
 \end{align*}
Substituting this into \eqref{l1}, we deduce (\ref{4.19}) from  \eqref{2.1}.
\end{proof}

\section{Super-convergence of Average-type \textrm{curl} Recovery Scheme in   Extended Linear Lagrange element}
  We define $\mathcal{N}_{h}$ as the nodal set of a quasi-uniform triangulation $\mathcal{T}_{h}$. Given $z \in \mathcal{N}_{h}$, we consider an element patch $\omega_z$ around $z$ where we choose as the origin of a local coordinates. Under this coordinate system, we let $z$ be the barycenter of a simplex $K_{j} \subset \omega_z, j=1,2, \ldots, m_z$, where $m_z$ is the number of elements containing the vertex $z$. 

Next we shall propose an average type curl recovery scheme, which is a similar version as  in \cite{huang}. 
Let us define the simple averaging operator $C_{h}:  \{\bm v\in\bm H_0(\mathrm{curl}, \Omega):\mathbf{curl}(\bm v|_{\overline K})\in \bm C^0(\overline K),~\forall K\in\mathcal T_h\} \mapsto \bm{L}_h$ as follows:

%
$$
C_{h}\bm u_{h}(z)=  \frac{1}{m_{z}} \sum_{K \subset \omega_{z}} \mathrm{curl} \bm u_{h}|_{K}.
$$
where  $\bm L_h$
is the tensorial linear Lagrange element space for $d=3$ (or the  linear Lagrange element space for $d=2$). 

%
%

Note that the recovery operator $C_{h}$ preserves $P_2$ polynomials at the node $z$:
\begin{align}
  (C_{h}  \bm p) (z)=(\mathrm{curl} \bm p) (z),\quad \forall\bm p\in \bm P_2(\omega_z).
\end{align}

Define $\bm {PW}^{3,\infty}(\Omega):=\{\bm v\in\bm L^{\infty}(\Omega):\bm v|_K\in \bm W^{3,\infty}(K),~ \forall K\in \mathcal T_h\}$.
\begin{thm} Let $T\bm f\in \bm {PW}^{3,\infty}(\Omega)$  then there holds the following supercloseness estimation for $C_{h}$: 
\begin{align}
\left\|\mathrm{curl}  T\bm f-C_{h}  T_{}\bm f\right\|_{} \lesssim h^{2}\|T\bm f\|_{\bm {PW}^{3,\infty}(\Omega)}.
\end{align}
\end{thm}
\begin{proof}
We decompose
$$
\mathrm{curl}T\bm f-C_{h} T_{}\bm f=\left(\mathrm{curl}T\bm f-(\mathrm{curl}T\bm f)_{I}\right)+\left((\mathrm{curl}T\bm f)_{I}-C_{h} T\bm f\right)
$$
where $(\mathrm{curl}T\bm f)_{I} \in \bm L_{h} $ is the piecewise linear interpolation of $\mathrm{curl}T\bm f$. By the standard approximation theory,
$$
\left\|\mathrm{curl} T\bm f-(\mathrm{curl}T\bm f)_{I}\right\|_{} \lesssim h^{2}\|T\bm f\|_{\bm {PW}^{3,\infty}(\Omega)}. 
$$
Moreover
$$
\begin{aligned}
\left\|(\mathrm{curl}T\bm f)_{I}-C_{h}T\bm f\right\|_{} & \leq\left(\sum_{K \in \Omega_{}}|K| \sum_{z \in \mathcal{N}_{h} \cap \bar{K}}\inf_{\bm q_z\in \bm P_2(\omega_z)}\left|C_{h} (T\bm f+\bm q_z)(z)-\mathrm{curl} (T\bm f+\bm q_z)(z)\right|^{2}\right)^{1 / 2} \\
& \lesssim h^{2}\|T\bm f\|_{\bm {PW}^{3,\infty}(\Omega)} \left|\Omega\right|^{1 / 2} \lesssim h^{2}\|T\bm f\|_{\bm {PW}^{3,\infty}(\Omega)}
\end{aligned}
$$
where $\bm q_z$ is an arbitrary quadratic polynomial in $\omega_z$ and $C_{h}  \bm q_z(z)=\mathrm{curl} (\bm q_z)(z)$. The combination of the above two estimates gives the conclusion.
\end{proof}

\section{Numerical Experiments}

In this section we provide numerical experiments using the extended linear and quadratic  Lagrange elements to support our theoretical work.   We generate the FE meshes as  the following approach.   For  the  cube and the thick L-shape domain in 3D, we partition the domains into small congruent
 cubes and then partition each
small cube into 12 tetrahedra.  For  the  tetrahedron in 3D, in  each mesh refinement we partition each tetrahedron in the mesh into 8 small tetrahedra. We refer the reader to \cite{chen} for the detail of the mesh refinement. For    the thick L-shape domain and the tetrahedron, the coarse  meshes and boundary degrees of freedom in extended quadratic Lagrange element are depicted in Figure 1.

\subsection{Numerical results for Maxwell equation}

First we shall consider the source problem \eqref{p2} of  Maxwell equations in $\Omega = (0,1)^3$ with the  exact solution

 \begin{equation}
         T\bm f= \begin{pmatrix}
           sin(\pi x_1)^3sin(\pi x_2)^2sin(\pi x_3)^2cos(\pi x_2) cos(\pi x_3) \\
           sin(\pi x_2)^3sin(\pi x_3)^2sin(\pi x_1)^2cos(\pi x_3)cos(\pi x_1) \\
           -2sin(\pi x_3)^3sin(\pi x_1)^2sin(\pi x_2)^2cos(\pi x_1)cos(\pi x_2)
          \end{pmatrix}.
 \end{equation}
 We compute the numerical solution $T_h\bm f$ and the error $\bm e_h:=T_h\bm f-T\bm f$.
 We also compute the recovery $ C_{h}T_h\bm f$ of $\mathrm{curl}T_h\bm f$ and denote the error $\delta CT_h\bm f:=C_{h}T_h\bm f-\mathrm{curl}T\bm f$.
 The errors of  numerical solutions obtained by the extended
linear  and quadratic  Lagrange  elements are listed in Table 1. It can be seen that the convergence rates in  the $\bm H(\mathrm{curl})$-seminorm and the $L^2$-norm are about 1 and 2 for the linear and quadratic  Lagrange  elements, respectively. It is also shown  from Table 1 that the recovery $ C_{h}T_h\bm f$ is a supercloseness to $\mathrm{curl}T_h\bm f$.


 \subsection{Numerical results for Maxwell eigenvalues}
 \indent Let $\{\bm \xi_i\}_{i=1}^{N_1}$ and $\{\bm \phi_i\}_{i=1}^{N_2}$ be the bases of $\bm L_{h,0}$ and $U_h$, respectively.
The eigenfunction in \eqref{equ:Deig} is written as $\bm u_h=\sum_{i=1}^{N_1}u_i\bm\xi_i+\sum_{i=1}^{N_2}u_{i+N_1}\nabla \phi_i$.
Denote $N_h:=N_1+N_2$ and $\vec{u}=(u_1,\cdots,u_{N_h})^T$.
  To describe our algorithm,
we specify
the following matrices in  the discrete case.\\
\begin{center} \footnotesize
\begin{tabular}{lllll}\hline
Matrix&Dimension&Definition\\\hline
$A_{11}$&$N_1\times N_1$&$A_{11}(l,i)=\int_\Omega \mathrm{curl}\bm\xi_i\cdot\mathrm{curl}\bm\xi_ldx$\\
$B_{11}$&$N_1\times N_1$&$B_{11}(l,i)=\int_\Omega  \bm\xi_i\cdot\bm\xi_ldx$\\
$B_{12}$&$N_1\times N_2$&$B_{12}(l,i)=\int_\Omega  \nabla\phi_l \cdot\bm\xi_idx$\\
$B_{22}$&$N_2\times N_2$&$B_{22}(l,i)=\int_\Omega  \nabla\phi_i \cdot\nabla\phi_ldx$\\
\hline
\end{tabular}
\end{center}
 Then  the discretization  can be written as the
generalized eigenvalue problems
\begin{align*}
A\vec u=\lambda B\vec u.
\end{align*}
where
\begin{eqnarray} \label{s5.1}
A:=\left(
\begin{array}{lcr}
A_{11}&0\\
0&0
\end{array}
\right)\quad\text{and}\quad
B:=\left(
\begin{array}{lcr}
B_{11}&B_{12}\\
B^T_{12}&B_{22}
\end{array}
\right).
\end{eqnarray}

In order to give the number of degrees of freedom (Dofs) for the extended Lagrange element methods, we use $N$, $NE$, $NF$, $NT$ to denote the number of nodes, edges, faces and tetrahedra in the mesh.
They satisfy the relations $N:NE:NF:NT\approx1:7:12:6$ in $\mathbb R^3$ and $N:NE:NF\approx1:3:2$ in $\mathbb R^2$.
We use NDofs and $\overline{\text{NDofs}}$ to denote the the number of Dofs in the extended Lagrange element methods and the second family of the N\'{e}d\'{e}lec element methods, respectively. For the comparative purpose they are listed  in the following table.

 \begin{center} \footnotesize
\begin{tabular}{llllll}\hline
&$p$&$Dim(\bm L_{h,0})$&$Dim(U_h)$&NDofs&$\overline{\text{NDofs}}$\\\hline
\multirow{2}*{$d=3$}&$1$&$3N$&$N+NE\approx 8N$&$ 11N$&$3NE \approx 14N$\\
&$2$&$3N+3NE\approx 24N$&$N+2NE+NF\approx 27N$&$ 52N$&$3NE + 3NF \approx57N$\\
\multirow{2}*{$d=2$}&1&$2N$&$N+NE\approx  4N$&$ 6N$&$2NE \approx 6N$\\
&2&$2N+2NE\approx 8N$&$N+2NE+NF\approx 9N$&$ 17N$&$3NE + 3NT \approx15N$\\
\hline
\end{tabular}
\end{center}

The above table shows that the computational complexity of the extended Lagrange element method in 3D is close to but less than the second family of N\'{e}d\'{e}lec element method in 3D when $p=1,2$.
\begin{remark}

If $\bm L_{h,0}\cap \nabla U_h\ne\{\bm0\}$, there is a nonzero vector $\vec u$ such that $\bm u_h=0$. Such $\vec u$ is exactly an eigenvector of \eqref{s5.1}. 
We can filter the algebraic spurious eigenvalue $\lambda_h$ using the following formula
$$\frac{\sqrt{|(\mathrm{curl}\bm u_h,\mathrm{curl}\bm u_h)-\lambda_h(\bm u_h,\bm u_h)|}}{\|\bm u_h\|}>\mathcal E$$
where $\mathcal E$ is the small tolerance.
\end{remark}

  We now test the new methods for solving the   Maxwell eigenvalue problem \eqref{2.10a}. The  first eight eigenvalues on different domains
  are listed as follows:

\begin{center}
\footnotesize
 \begin{tabular}{ccccccccc}
  \hline
Unit Square: &$\pi^2$&$\pi^2$&$2\pi^2$&$4\pi^2$&$4\pi^2$&$5\pi^2$&$5\pi^2$&$8\pi^2$\\\hline
L-shape:& 1.4756218&3.5340314&$\pi^2$&$\pi^2$&11.389479&12.57219&$2\pi^2$&21.4242598\\\hline
Unit Cube:& $2\pi^2$&$2\pi^2$&$2\pi^2$&$3\pi^2$&$3\pi^2$&$5\pi^2$&$5\pi^2$&$5\pi^2$\\\hline
Thick L-shape:&9.63972& 11.34523&
13.40364& 15.19725& 19.50933&$2\pi^2$&$2\pi^2$&$2\pi^2$
  \\\hline
 \end{tabular}
\end{center}
 Let $(\lambda_{i},\bm u_{i})$ be the $i$th Maxwell eigenpair and $(\lambda_{i,h},\bm u_{i,h})$ be its approximation obtained by the extended Lagrange FEM.  From   \cite{babuska} we have
\begin{align}
\lambda_{i,h}-\lambda_i=\frac{(\mathrm{curl}(\bm u_{i,h}-\bm u_i),\mathrm{curl}(\bm u_{i,h}-\bm u_i))}{(\bm u_{i,h},\bm u_{i,h})}-\lambda\frac{(\bm u_{i,h}-\bm u_i,\bm u_{i,h}-\bm u_i)}{(\bm u_{i,h},\bm u_{i,h})},\quad i=1,2,\cdots.
\end{align}
If the eigenfunction $\bm u$   is  smooth,   it is expected that $\lambda_{i,h}\ge\lambda_i$ because $\|\bm u_h-\bm u\|$ is of higher order than $\|\mathrm{curl}(\bm u_h-\bm u)\|$.

We first use the extended linear Lagrange FEM to  compute for the lowest eight eigenvalues.
It can be seen from Tables 2-6 that for the case of  the unit square in 2D (c.f. Table 2) and the cube,  the   tetrahedron and the thick L-shape in 3D(c.f. Tables 4-6), all computed eigenvalues obtained by  extended
linear  Lagrange  FEMs present the upper bounds of the associated eigenvalues; so do the numerical eigenvalues $\lambda_{i,h}$ $(i=2,\cdots,8)$ on the L-shape in 2D; however $\lambda_{1,h}$ obtained by the extended linear Lagrange FEM  on the L-shape domain does not  because the associated eigenfunction $\bm u_1$ is strongly unbounded singular.  The same is   for $\lambda_{2,h}$ obtained by the extended quadratic Lagrange FEM.
In all cases given here the  convergence rates of the computed eigenvalues by extended linear Lagrange FEM are around 2, which are optimal.

We next use the extended quadratic Lagrange FEM to compute for the lowest eight eigenvalues. It can be seen from Tables 2-6 that  for the case of the unit square in 2D(c.f. Table 2)
and the cube, the tetrahedron and the thick L-shape in 3D(c.f. Tables 4-6), all computed eigenvalues obtained by  extended
 quadratic Lagrange  FEMs present the upper bounds of the associated eigenvalues; so do the numerical eigenvalues $\lambda_{i,h}$$(i=3,4,5,7)$ on the L-shape in 2D(c.f. Table 3); however $\lambda_{i,h}$$(i=1,2,6,8)$ obtained by the extended quadratic Lagrange FEM  on the L-shape domain does not  because of the singularity of  the associated eigenfunctions.

 The   convergence rates of the computed eigenvalues by extended quadratic Lagrange FEM on the square in 2D(c.f. Table 2) and the cube and tetrahedron in 3D(c.f. Tables 4-5) are around 4, which are optimal.
However,  this is not true for the computed  eigenvalue $\lambda_{i,h}$$(i=1,2,6,8)$   obtained by the extended quadratic Lagrange FEM on the L-shape domain in 2D(c.f. Table 3) and the computed  eigenvalue $\lambda_{i,h}$$(i=1,2,5)$ obtained by the extended quadratic Lagrange FEM on the thick L-shape domain in 3D(c.f. Table 6). This is due to the singularity of the associated eigenfunctions.

\subsection{Supercloseness lower bound by recovery approach}

Using the recovery approach in \cite{naga}, for the numerical
eigenpair $(\lambda_{i,h},\bm u_{i,h})$ obtained by the extended linear Lagrange FEM, we shall  compute the  a-posteriori error estimator of recovery type
$$\eta_h(\bm u_{i,h}):=\frac{\|\mathrm{curl}\bm u_{i,h}-C_{h}\bm u_{i,h}\|^2}{\|\bm{u}_{i,h}\|^2},\quad i=1,2,\cdots.$$
Let $\tilde\lambda_{i,h}:=\lambda_{i,h}-\eta_h(\bm u_{i,h})$. It is expected that $\tilde\lambda_{i,h}$ has the supercloseness property when the eigenfunction $\bm u_i$ is smooth.

For the case of the unit square in 2D and the cube in 3D, it can be seen from Tables 7 and 9 that  $\tilde\lambda_{i,h}$ ($i=1,2,3$) are   supercloseness lower bounds of $\lambda_{i}$($i=1,2,3$), respectively.

For the case of the  L-shape domain in 2D, we can see from Table 8 that $\tilde\lambda_{i,h}$ ($i=2,3,4$) are   supercloseness lower bounds of $\lambda_{i}$($i=2,3,4$), respectively;  we can see from Table 3 that $\lambda_{1,h}$ does not present an upper bound of $\lambda_1$ because the associated eigenfunction $\bm u_1$ is strongly unbounded singular. For this reason  the supercloseness property  of $\tilde\lambda_{1,h}$ is invalid and  we do not list $\tilde\lambda_{1,h}$.

For the case of the thick L-shape in 3D, we can see from Table 9 that $\tilde\lambda_{i,h}$ ($i=1,2,3$) are    lower bounds of $\lambda_{i}$ ($i=1,2,3$), respectively,  and meanwhile $\tilde\lambda_{3,h}$ is   a supercloseness of $\lambda_3$. $\tilde\lambda_{1,h}$ and $\tilde\lambda_{2,h}$ do not supercloseness property due to the singularity of the eigenfunctions $\bm u_1$ and $\bm u_2$.

\begin{table}[htbp]\footnotesize
 \caption{Error for the solution in \eqref{p2} on the unit cube using extended linear (left) and quadratic (right) Lagrange finite element methods.}
 \centering
 \begin{tabular}{ccccccc}
  \hline
$h$&  1/2&	 1/4&	 1/6&	 1/8&	 1/10\\\hline
$\|\mathrm{curl}(\bm e_h)\|$&0.7516&	0.4818&	0.3351&	0.2556&	0.2062\\$Rate$&--&	0.64& 	0.90& 	0.94& 	0.96\\
$\|\bm e_h\|$&0.0555&	0.0244&	0.0117&	0.0067&	0.0044\\$Rate$&--&	1.19& 	1.81& 	1.94& 	1.88\\
$\|\delta C\mathrm{curl}T_h\bm f\|$&0.8063&	0.3778&	0.1363&	0.056&	0.0266\\$Rate$&--&		1.09& 	2.51& 	3.09& 	3.34
  \\\hline
 \end{tabular}
 ~
 \begin{tabular}{ccccccc}
  \hline
  1/2&	 1/3&	 1/4&	 1/5&1/6\\\hline
0.425&	0.2291&	0.1168&	0.0769&	0.0544\\--	&1.52& 	2.34& 	1.87& 	1.90\\
0.017&	0.0059&	0.0023&	0.0012&	7.38E-04\\--	&2.61& 	3.27& 	2.92& 	2.67
  \\\hline
  \\\\
 \end{tabular}
 \label{PG}
 \end{table}

\begin{table}[htbp]\footnotesize
 \caption{Numerical eigenvalues on the unit square obtained by extended linear (left) and quadratic (right) Lagrange finite element methods.}
 \centering
 \begin{tabular}{ccccccc}
  \hline
$h$&    1/4& 	  1/8& 	  1/16&	  1/32\\\hline
$\lambda_{1,h}$&10.0918& 	9.92542& 	9.88357& 	9.87310\\
$Rate$&--&1.99$\downarrow$& 	2.00$\downarrow$& 	2.00$\downarrow$\\
$\lambda_{2,h}$&10.1067$\downarrow$& 	9.92907$\downarrow$& 	9.88448$\downarrow$& 	9.87332$\downarrow$\\
$Rate$&--&2.00$\downarrow$& 	2.00$\downarrow$& 	2.00$\downarrow$\\
$\lambda_{3,h}$&20.7531& 	19.9912& 	19.8021& 	19.7549\\
$Rate$&--&2.01$\downarrow$& 	2.00$\downarrow$& 	2.00$\downarrow$\\
$\lambda_{4,h}$&43.1269& 	40.4107& 	39.7121& 	39.5368\\
$Rate$&--&1.97$\downarrow$& 	2.00$\downarrow$& 	2.00$\downarrow$\\
$\lambda_{5,h}$&43.3726& 	40.4638& 	39.7252& 	39.5401\\
$Rate$&--&1.98$\downarrow$& 	2.00$\downarrow$& 	2.00$\downarrow$\\
$\lambda_{6,h}$&55.4626& 	50.8812& 	49.7304& 	49.4435\\
$Rate$&--&2.00$\downarrow$& 	2.00$\downarrow$& 	2.00$\downarrow$\\
$\lambda_{7,h}$&55.9247& 	50.9809& 	49.7541& 	49.4494\\
$Rate$&--&2.01$\downarrow$& 	2.01$\downarrow$& 	2.00$\downarrow$\\
$\lambda_{8,h}$&94.8772& 	83.0772& 	79.9849& 	79.2136\\
$Rate$&--&1.95$\downarrow$& 	2.00$\downarrow$& 	2.00$\downarrow$
  \\\hline
 \end{tabular}
 ~
 \begin{tabular}{ccccccc}
  \hline
    1/4& 	  1/8& 	  1/16&	  1/32\\\hline
9.8712349&	9.8697136&	9.8696108&	9.8696048\\--&	3.90$\downarrow$ &	4.08$\downarrow$ 	&4.00$\downarrow$\\
9.8713258&	9.8697073&	9.8696112&	9.8696048\\--&	4.06$\downarrow$ &	3.91$\downarrow$ &	4.00$\downarrow$\\
19.750518&	19.73995&	19.739256&	19.739211\\--&	3.92$\downarrow$& 	3.98$\downarrow$ 	&4.00$\downarrow$\\
39.574996&	39.484880&	39.478789&	39.478443\\--&	3.90$\downarrow$ &	4.12$\downarrow$ &	3.85$\downarrow$\\
39.566074&	39.484259&	39.478828&	39.478440\\--&	3.91$\downarrow$& 	3.83$\downarrow$ &	4.14$\downarrow$\\
49.514075	&49.359361&	49.348848	&49.348067\\--&	3.87$\downarrow$& 	3.78$\downarrow$ &	4.18$\downarrow$\\
49.534035&	49.360921&	49.348748	&49.348074\\--&		3.85$\downarrow$ &	4.15$\downarrow$ &	3.80$\downarrow$\\
79.639538&	79.005160&	78.959969	&78.957032\\--&		3.82$\downarrow$& 	3.95$\downarrow$ &	3.99$\downarrow$
  \\\hline
 \end{tabular}
 \label{PG1}
 \end{table}

 \begin{table}[htbp]\footnotesize
 \caption{Numerical eigenvalues on the L-shape domain obtained by extended linear (left) and quadratic (right) Lagrange finite element methods.}
 \centering
 \begin{tabular}{ccccccc}
  \hline
$h$&      1/4& 	  1/8& 	  1/16&	  1/32&	  1/64\\\hline
$\lambda_{1,h}$&1.4777& 	1.47511& 	1.47508& 	1.47532& 	1.47548\\
$Rate$&--&2.49$\downarrow$& 	-0.12$\downarrow$& 	1.17$\uparrow$& 	2.52$\uparrow$\\
$\lambda_{2,h}$&3.5694& 	3.54292& 	3.53630& 	3.53459& 	3.53417\\
$Rate$&--&2.00$\downarrow$& 	1.97$\downarrow$& 	2.02$\downarrow$& 	2.00$\downarrow$\\
$\lambda_{3,h}$&10.095& 	9.92637& 	9.88380& 	9.87316& 	9.87049\\
$Rate$&--&1.99$\downarrow$& 	2.00$\downarrow$& 	2.00$\downarrow$& 	2.00$\downarrow$\\
$\lambda_{4,h}$&10.132& 	9.93564& 	9.88610& 	9.87374& 	9.87064\\
$Rate$&--&1.99$\downarrow$& 	2.00$\downarrow$& 	2.00$\downarrow$& 	2.00$\downarrow$\\
$\lambda_{5,h}$&11.699& 	11.4677& 	11.4091& 	11.3943& 	11.3907\\
$Rate$&--&1.98$\downarrow$& 	2.00$\downarrow$& 	2.00$\downarrow$& 	2.00$\downarrow$\\
$\lambda_{6,h}$&12.933& 	12.6608& 	12.5933& 	12.5770& 	12.5733\\
$Rate$&--&2.03$\downarrow$& 	2.07$\downarrow$& 	2.11$\downarrow$& 	2.08$\downarrow$\\
$\lambda_{7,h}$&20.758& 	19.9959& 	19.8035& 	19.7552& 	19.7432\\
$Rate$&--&1.99$\downarrow$& 	2.00$\downarrow$& 	2.00$\downarrow$& 	2.00$\downarrow$\\
$\lambda_{8,h}$&22.598& 	21.7148& 	21.4944& 	21.4409& 	21.4282\\
$Rate$&--&2.01$\downarrow$& 	2.05$\downarrow$& 	2.07$\downarrow$& 	2.05$\downarrow$
  \\\hline
 \end{tabular}
 ~
 \begin{tabular}{ccccccc}
  \hline
  1/4 &	  1/8& 	  1/16&	  1/32	\\\hline			
1.47278&	1.474489	&1.475171&	1.475443\\--&		1.41$\uparrow$& 	1.58$\uparrow$ &	2.27$\uparrow$\\
3.53407	&3.534029&	3.5340305	&{ 3.5340312}	\\--&	4.63$\downarrow$ &	1.08$\uparrow$ &	2.32$\uparrow$\\
9.87118	&9.869704&	9.8696106	&{ 9.8696048}	\\--&	3.99$\downarrow$ &	4.00$\downarrow$& 	4.00$\downarrow$\\
9.87201	&9.869756&	9.8696139	&{ 9.8696050}	\\--&	3.98$\downarrow$ 	&3.99$\downarrow$ 	&4.00$\downarrow$\\
11.3926	&11.38967&	11.38949	&11.38947	\\--&	4.03$\downarrow$ &	4.12$\downarrow$ &	4.36$\downarrow$\\
12.5674	&12.56899&	12.57094	&12.57181	\\--&	0.58$\uparrow$ &	1.36$\uparrow$ &	1.71$\uparrow$\\
19.7525	&19.74008&	19.73926	&19.73921	\\--&	3.92$\downarrow$ 	&3.98$\downarrow$ 	&4.00$\downarrow$\\
21.4232	&21.41747&	21.42140	&21.42338	\\--&	-2.74$\downarrow$& 	1.25$\uparrow$ &	1.70$\uparrow$
  \\\hline
 \end{tabular}
 \label{PG2}
 \end{table}

 \begin{table}[htbp]\footnotesize
 \caption{Numerical eigenvalues on the unit cube obtained by extended linear (left) and quadratic (right) Lagrange finite element methods.}
 \centering
 \begin{tabular}{ccccccc}
  \hline
$h$&  1/2&	 1/3&	 1/4&	 1/5&	 1/6\\\hline
$\lambda_{1,h}$&24.6853&	21.6648&	20.7846&	20.3917&	20.1865\\
$Rate$&--&2.33$\downarrow$&	2.12$\downarrow$&	2.11$\downarrow$&	2.07$\downarrow$\\
$\lambda_{2,h}$&24.7298&	21.7094&	20.7907&	20.3935&	20.1886\\
$Rate$&--&2.29$\downarrow$&	2.18$\downarrow$&	2.13$\downarrow$&	2.06$\downarrow$\\
$\lambda_{3,h}$&24.7577&	21.7018&	20.7935&	20.396&	20.1902\\
$Rate$&--&2.32$\downarrow$&	2.16$\downarrow$&	2.12$\downarrow$&	2.06$\downarrow$\\
$\lambda_{4,h}$&39.9657&	33.1761&	31.4867&	30.7537&	30.3871\\
$Rate$&--&2.63$\downarrow$&	2.23$\downarrow$&	2.22$\downarrow$&	2.12$\downarrow$\\
$\lambda_{5,h}$&39.9684&	33.301&	31.5064&	30.7633&	30.396\\
$Rate$&--&2.54$\downarrow$&	2.31$\downarrow$&	2.23$\downarrow$&	2.1$\downarrow$\\
$\lambda_{6,h}$&74.4975&	61.1531&	56.009&	53.5104&	52.2029\\
$Rate$&--&1.87$\downarrow$&	1.99$\downarrow$&	2.11$\downarrow$&	2.07$\downarrow$\\
$\lambda_{7,h}$&74.6583&	61.356&	56.0925&	53.539&	52.2052\\
$Rate$&--&1.84$\downarrow$&	2.01$\downarrow$&	2.13$\downarrow$&	2.10$\downarrow$\\
$\lambda_{8,h}$&74.7468&	61.7022&	56.1204&	53.5613&	52.2208\\
$Rate$&--&1.78$\downarrow$&	2.09$\downarrow$&	2.13$\downarrow$&	2.10$\downarrow$
  \\\hline
 \end{tabular}
 ~
 \begin{tabular}{ccccccc}
  \hline
  1/2&	 1/3&	 1/4&	 1/5&1/6\\\hline
19.879&	19.7648&	19.7468&	19.7421&	19.7406\\--
&4.19$\downarrow$&	4.22$\downarrow$&	4.33$\downarrow$&	4.01$\downarrow$\\
19.8836&	19.7659&	19.7468&	19.7422&	19.7406\\--
&4.16$\downarrow$&	4.37$\downarrow$&	4.17$\downarrow$&	4.2$\downarrow$\\
19.886&	19.7659&	19.7469&	19.7422&	19.7406\\--
&4.20$\downarrow$&	4.33$\downarrow$&	4.23$\downarrow$&	4.20$\downarrow$\\
30.0109&	29.6946&	29.6345&	29.6189&	29.6134\\--
&3.81$\downarrow$&	4.19$\downarrow$&	4.19$\downarrow$&	4.32$\downarrow$\\
30.0139&	29.6966&	29.6347&	29.619&	29.6134\\--
&3.77$\downarrow$&	4.24$\downarrow$&	4.18$\downarrow$&	4.38$\downarrow$\\
50.9636&	49.6862&	49.4548&	49.3926&	49.3697\\--
&3.86$\downarrow$&	4.01$\downarrow$&	3.91$\downarrow$&	3.95$\downarrow$\\
51.0757&	49.6884&	49.4561&	49.3931&	49.3699\\--
&4.01$\downarrow$&	3.99$\downarrow$&	3.92$\downarrow$&	3.97$\downarrow$\\
51.1667&	49.6958&	49.4577&	49.3935&	49.3699\\--
&4.08$\downarrow$&	4.01$\downarrow$&	3.95$\downarrow$&	4.01$\downarrow$
  \\\hline
 \end{tabular}
 \label{PG}
 \end{table}

 \begin{table}[htbp]\footnotesize
 \caption{Numerical eigenvalues on the tetrahedron obtained by extended linear (left) and quadratic (right) Lagrange finite element methods.}
 \centering
 \begin{tabular}{ccccccc}
  \hline
$h$&  1/2&	 1/4&	 1/8&	 1/16\\\hline
$\lambda_{1,h}$&32.3467&	27.6752&	26.46&	26.156\\
$Rate$&--&	1.96$\downarrow$& 	2.00$\downarrow$& 	2.01$\downarrow$\\
$\lambda_{2,h}$&33.079&	27.7695&	26.4806&	26.161\\
$Rate$&--&2.04$\downarrow$& 	2.01$\downarrow$& 	2.01$\downarrow$\\
$\lambda_{3,h}$&33.079&	27.7695&	26.4806&	26.161\\
$Rate$&&2.04$\downarrow$& 	2.01$\downarrow$& 	2.01$\downarrow$\\
$\lambda_{4,h}$&80.8141&	60.8229&	55.2449&	53.8868\\
$Rate$&--&1.89$\downarrow$& 	2.04$\downarrow$& 	2.03$\downarrow$\\
$\lambda_{5,h}$&81.3878&	61.3432&	55.4132&	53.9318\\
$Rate$&--&1.82$\downarrow$& 	2.01$\downarrow$& 	2.02$\downarrow$\\
$\lambda_{6,h}$&81.3878&	61.3432&	55.4132&	53.9318\\
$Rate$&--&1.82$\downarrow$& 	2.01$\downarrow$& 	2.02$\downarrow$\\
$\lambda_{7,h}$&81.5981&	63.3483&	58.5518&	57.3607\\
$Rate$&--&1.95$\downarrow$& 	2.01$\downarrow$& 	2.03$\downarrow$\\
$\lambda_{8,h}$&85.5942&	64.6753&	58.9099&	57.4519\\
$Rate$&--&1.89$\downarrow$& 	1.99$\downarrow$& 	2.02$\downarrow$
  \\\hline
 \end{tabular}
 ~
 \begin{tabular}{ccccccc}
  \hline
  1/4&	 1/8&	 1/16&	 1/32\\\hline
26.532&	26.2471&	26.0698&	26.0557\\--&	1.32$\downarrow$&	3.76$\downarrow$&	3.65$\downarrow$\\
27.7558&	26.2883&	26.0726&	26.0559\\--&	2.87$\downarrow$&	3.8$\downarrow$&	3.69$\downarrow$\\
27.7558&	26.2883&	26.0726&	26.0559\\--&	2.87$\downarrow$&	3.8$\downarrow$&	3.69$\downarrow$\\
70.0000&	54.9355&	53.5601&	53.446\\--&	3.47$\downarrow$&	3.71$\downarrow$&	3.59$\downarrow$\\
70.8989&	54.9545&	53.5859&	53.448\\--&	3.53$\downarrow$&	3.45$\downarrow$&	3.31$\downarrow$\\
74.6667&	54.9545&	53.5859&	53.448\\--&	3.82$\downarrow$&	3.45$\downarrow$&	3.31$\downarrow$\\
76.6727&	58.6841&	57.1064&	56.9740\\--&	3.53$\downarrow$&	3.69$\downarrow$&	3.58$\downarrow$\\
76.6727&	58.8057&	57.1416&	56.9769\\--&	3.43$\downarrow$&	3.47$\downarrow$&	3.34$\downarrow$
  \\\hline
 \end{tabular}
 \label{PG}
 \end{table}

 \begin{table}\footnotesize
 \caption{Numerical eigenvalues on the thick L-shape domain obtained by extended linear (left) and quadratic (right) Lagrange finite element methods.}
 \centering
 \begin{tabular}{ccccccc}
  \hline
$h$&  1/2&	 1/3&	 1/4&	 1/5&	 1/6\\\hline
$\lambda_{1,h}$&11.091&	10.341&	10.0498&	9.9144&	9.8393\\$Rate$&--&	1.79$\downarrow$&	1.87$\downarrow$&	1.80$\downarrow$&	1.75$\downarrow$\\
$\lambda_{2,h}$&14.0532&	12.5497&	12.0375&	11.7989&	11.6686\\$Rate$&--&	2.00$\downarrow$&	1.93$\downarrow$&	1.89$\downarrow$&	1.86$\downarrow$\\
$\lambda_{3,h}$&15.3133&	14.2097&	13.8451&	13.6809&	13.5949\\$Rate$&--&	2.13$\downarrow$&	2.09$\downarrow$&	2.08$\downarrow$&	2.04$\downarrow$\\
$\lambda_{4,h}$&17.6559&	16.2547&	15.7792&	15.5624&	15.4465\\$Rate$&--&	2.08$\downarrow$&	2.08$\downarrow$&	2.09$\downarrow$&	2.09$\downarrow$\\
$\lambda_{5,h}$&23.9564&	21.6577&	20.7545&	20.3179&	20.0812\\$Rate$&--&	1.79$\downarrow$&	1.9$\downarrow$&	1.93$\downarrow$&	1.9$\downarrow$\\
$\lambda_{6,h}$&24.046&	21.6728&	20.7908&	20.3998&	20.1912\\$Rate$&--&	1.98$\downarrow$&	2.12$\downarrow$&	2.08$\downarrow$&	2.08$\downarrow$\\
$\lambda_{7,h}$&24.4227&	21.7218&	20.7969&	20.4008&	20.1916\\$Rate$&--&	2.12$\downarrow$&	2.18$\downarrow$&	2.10$\downarrow$&	2.08$\downarrow$\\
$\lambda_{8,h}$&24.7486&	21.7637&	20.8081&	20.4013&	20.1924\\$Rate$&--&	2.23$\downarrow$&	2.22$\downarrow$&	2.15$\downarrow$&	2.08$\downarrow$
  \\\hline
 \end{tabular}
 ~
 \begin{tabular}{ccccccc}
  \hline
  1/2&	 1/3&	 1/4&	 1/5&1/6\\\hline
9.7357&	9.6882&	9.6717&	9.6632&	9.658\\--&1.68$\downarrow$&	1.45$\downarrow$&	1.38$\downarrow$&	1.37$\downarrow$\\
11.4258&	11.3765&	11.3635&	11.3579&	11.3548\\--&	2.33$\downarrow$&	1.87$\downarrow$&	1.64$\downarrow$&	1.54$\downarrow$\\
13.4498&	13.4121&	13.4063&	13.4048&	13.4042\\--&	4.18$\downarrow$&	4.02$\downarrow$&	3.71$\downarrow$&	3.97$\downarrow$\\
15.2731&	15.2126&	15.2025&	15.1996&	15.1985\\--&	3.94$\downarrow$&	3.73$\downarrow$&	3.6$\downarrow$&	3.47$\downarrow$\\
19.7751&	19.6082&	19.5667&	19.5492&	19.5393\\--&	2.44$\downarrow$&	1.89$\downarrow$&	1.63$\downarrow$&	1.57$\downarrow$\\
19.8665&	19.7619&	19.7461&	19.7419&	19.7405\\--&	4.25$\downarrow$&	4.14$\downarrow$&	4.21$\downarrow$&	4.03$\downarrow$\\
19.869&	19.7626&	19.7461&	19.7419&	19.7405\\--&	4.23$\downarrow$&	4.25$\downarrow$&	4.21$\downarrow$&	4.03$\downarrow$\\
19.8817&	19.7631&	19.7463&	19.7421&	19.7405\\--&	4.40$\downarrow$&	4.22$\downarrow$&	4.02$\downarrow$&	4.42$\downarrow$
  \\\hline
 \end{tabular}
 \label{PG}
 \end{table}

 \begin{table}\footnotesize
 \caption{Recovered discrete eigenvalues on the square.}
 \centering
 \begin{tabular}{ccccccc}
  \hline
$h$&   1/4 &	  1/8 &	  1/16&	  1/32&	  1/64\\\hline
$\tilde\lambda_{1,h}$&9.79974& 	9.86602& 	9.86933& 	9.86958& 	9.86960\\$Rate$&--&4.28$\uparrow$& 	3.71$\uparrow$& 	3.56$\uparrow$& 	3.42$\uparrow$\\
$\tilde\lambda_{2,h}$&19.24442&	19.70742&	19.73691&	19.73903&	19.73919\\$Rate$&--&3.96$\uparrow$& 	3.79$\uparrow$& 	3.69$\uparrow$ &	3.56$\uparrow$\\
$\tilde\lambda_{3,h}$&35.99717&	39.21930&	39.46352&	39.47736&	39.47834\\$Rate$&--&3.75$\uparrow$ &	4.12$\uparrow$& 	3.82$\uparrow$& 	3.79$\uparrow$
  \\\hline
 \end{tabular}
 \end{table}
 \begin{table}\footnotesize
 \caption{Recovered discrete eigenvalues on the   L-shape domain.}
 \centering
 ~
 \begin{tabular}{ccccccc}
  \hline
$h$& 	  1/4& 	  1/8& 	  1/16&	  1/32&	  1/64\\\hline
$\tilde\lambda_{2,h}$&3.532018&	3.533785&	3.534006&	3.534028&	3.534031\\$Rate$	&--&3.03$\uparrow$& 	3.25$\uparrow$& 	3.10$\uparrow$& 	2.96$\uparrow$\\
$\tilde\lambda_{3,h}$&9.82203&	9.86564&	9.86929&	9.86958&	9.86960\\$Rate$&--&3.59$\uparrow$& 	3.64$\uparrow$& 	3.54$\uparrow$& 	3.40$\uparrow$\\
$\tilde\lambda_{4,h}$&11.29961&	11.38192&	11.38890&	11.38943&	11.38948\\$Rate$&--&3.57$\uparrow$& 	3.70$\uparrow$& 	3.64$\uparrow$& 	3.51$\uparrow$
  \\\hline
 \end{tabular}
 \label{PG}
 \end{table}

 \begin{table}\footnotesize
 \caption{Recovered discrete eigenvalues on the cube (left) and the thick L-shape domain (right).}
 \centering
 \begin{tabular}{ccccccc}
  \hline
$h$&   	   1/4& 	  1/8& 	  1/16&	  1/32&	  1/64\\\hline
$\tilde\lambda_{1,h}$&9.79974& 	9.86602& 	9.86933& 	9.86958& 	9.86960\\$Rate$&--&	4.28$\uparrow$& 	3.71$\uparrow$& 	3.56$\uparrow$& 	3.42$\uparrow$\\
$\tilde\lambda_{2,h}$&19.24442&	19.70742&	19.73691&	19.73903&	19.73919\\$Rate$&--&3.96$\uparrow$& 	3.79$\uparrow$& 	3.69$\uparrow$& 	3.56$\uparrow$\\
$\tilde\lambda_{3,h}$&35.99716&	39.21929&	39.46352&	39.47736&	39.47834\\$Rate$&--&3.75$\uparrow$& 	4.12$\uparrow$& 	3.82$\uparrow$& 	3.79$\uparrow$
  \\\hline
 \end{tabular}
 ~
 \begin{tabular}{ccccccc}
  \hline
 	 1/3&	 1/4&	 1/5&1/6\\\hline
9.345&	9.5008&	9.5555&	9.5826\\	--&2.61$\uparrow$& 	2.24$\uparrow$& 	2.13$\uparrow$\\
11.0998&	11.2155&	11.2574&	11.2809\\--&	2.22$\uparrow$& 	1.75$\uparrow$& 	1.71$\uparrow$\\
13.1075&	13.2933&	13.3529&	13.3767\\--&	3.43$\uparrow$& 	3.48$\uparrow$& 	3.47$\uparrow$
  \\\hline
 \end{tabular}
 \label{PG}
 \end{table}

\begin{figure}[htbp]\label{mesh}
\centering
\includegraphics[width=0.3\textwidth,trim=50 25 25 50]{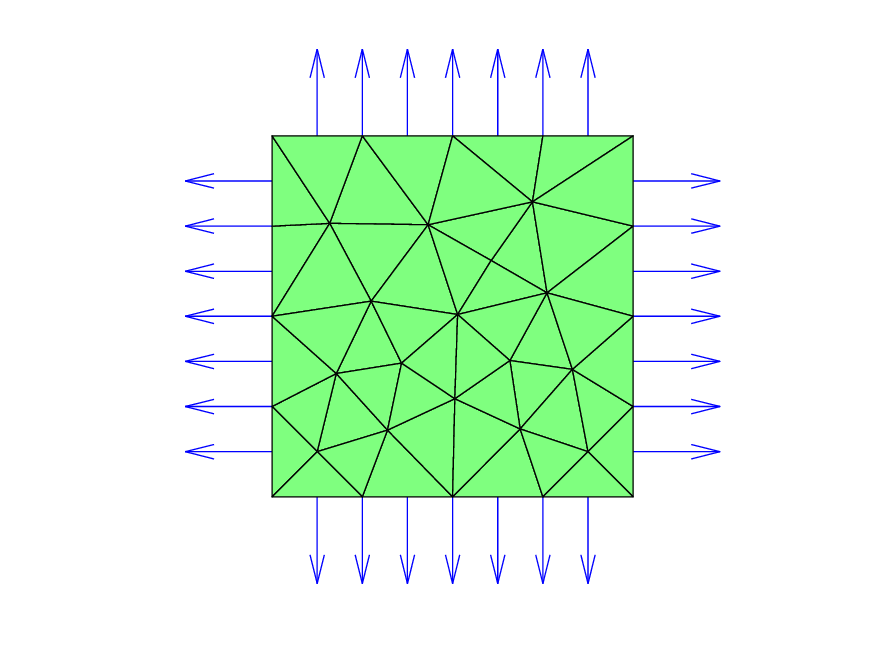}
\includegraphics[width=0.4\textwidth]{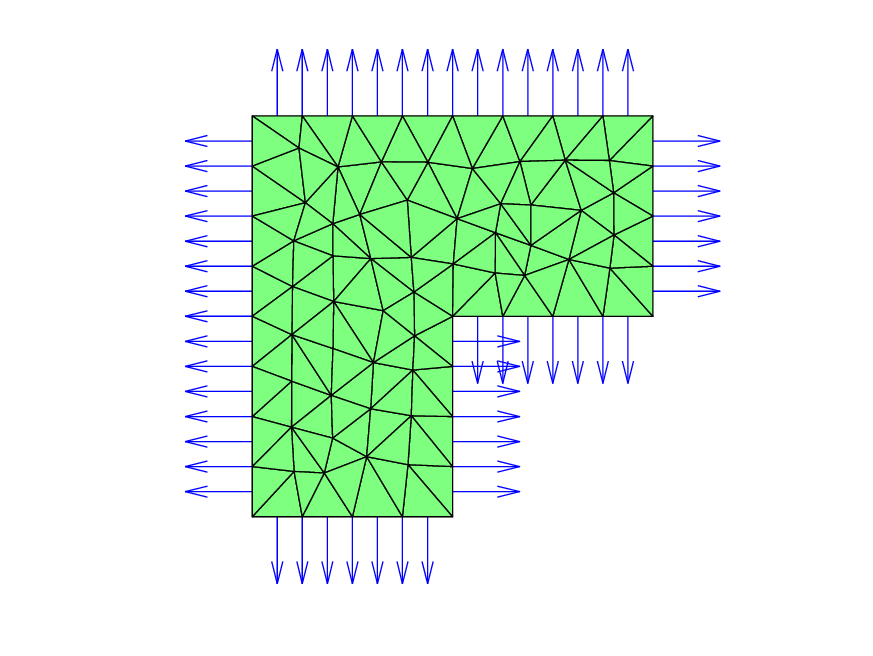}\\
\includegraphics[width=0.35\textwidth]{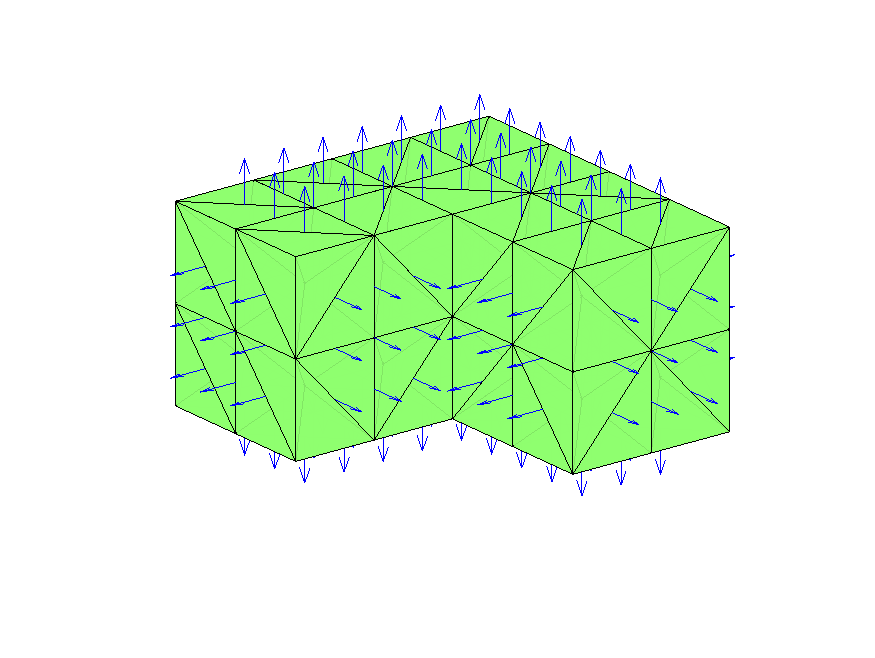}
\includegraphics[width=0.35\textwidth,trim=35 35 50 50,clip]{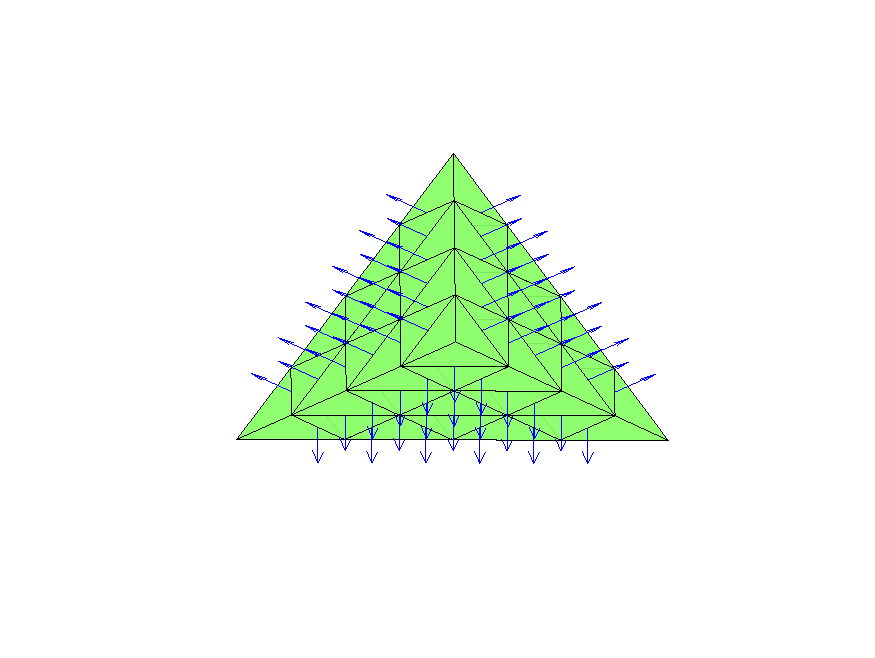}
\caption{Coarse meshes and degrees of freedom on $\partial\Omega$   for extended quadratic Lagrange elements in 2D and 3D.}
\end{figure}




\SG{ \section{Conclusion}
In this paper, we proposed a new method using Lagrange element space to solve
the  Maxwell equations and the associated eigenvalue problem. 
To obtain numerical solution of supercloseness, we propose an average type curl recovery scheme for the extended linear lagrange FEM. This scheme can be used to obtain the numerical lower bound of Maxwell eigenvalues.
Although we only focus on the Maxwell equation and its eigenvalue problem with constant coefficients,  the new method and the associated theoretical results are suitable for  discontinuous coefficients.
}

 \appendix


\bibliographystyle{unsrt}
\bibliography{literature}
\end{document}